\documentclass[10pt,psamsfonts]{amsart}
\usepackage[utf8]{inputenc}

\usepackage{amsmath}
\usepackage{amsthm}
\usepackage{amssymb}
\usepackage{amscd}
\usepackage{amsfonts}
\usepackage{amsbsy}
\usepackage{graphicx}
\usepackage[dvips]{psfrag}
\usepackage{array}
\usepackage{color}
\usepackage{epsfig}
\usepackage{url}
\usepackage{enumerate}
\newcommand{\Int}{\mathrm{Int}}
\usepackage[top=3cm,bottom=2cm,left=1.5cm,right=1.5cm]{geometry}
	
\usepackage{overpic}
\usepackage{epstopdf}

\newcommand{\R}{\ensuremath{\mathbb{R}}}

\newcommand{\ov}{\overline}

\newcommand{\x}{\mathbf{x}}

\newcommand{\sgn}{\mathrm{sign}}

\newcommand{\bb}{\mathbf{b}}

\def\p{\partial}
\def\e{\varepsilon}

\newtheorem {theorem} {Theorem}%[section]
\newtheorem {definition}[theorem] {Definition}
\newtheorem {proposition} {Proposition}

\newtheorem {lemma}[theorem] {Lemma}

\newtheorem {remark} {Remark}

\usepackage{mathpazo}

\begin{document}
\renewcommand{\arraystretch}{1.5}

\title[On the number of limit cycles of planar piecewise linear differential systems]
{Uniform upper bound for the number of limit cycles\\ of planar piecewise linear differential systems\\ with two zones separated by a straight line}

\author[V. Carmona, F. Fernandez-S\'{a}nchez, and D. D. Novaes]
{Victoriano Carmona$^1$, Fernando Fern\'{a}ndez-S\'{a}nchez$^2$,\\ and Douglas D. Novaes$^3$}

\address{$^1$ Dpto. Matem\'{a}tica Aplicada II \& IMUS, Universidad de Sevilla, Escuela Polit\'ecnica Superior.
Calle Virgen de \'Africa 7, 41011 Sevilla, Spain.} 
 \email{vcarmona@us.es} 

\address{$^2$ Dpto. Matem\'{a}tica Aplicada II \& IMUS, Universidad de Sevilla, Escuela T\'{e}cnica Superior de Ingenier\'{i}a.
Camino de los Descubrimientos s/n, 41092 Sevilla, Spain.} \email{fefesan@us.es}

\address{$^3$ Departamento de Matem\'{a}tica, Instituto de Matem\'{a}tica, Estatística e Computa\c{c}\~{a}o Cient\'{i}fica (IMECC), Universidade
Estadual de Campinas (UNICAMP), Rua S\'{e}rgio Buarque de Holanda, 651, Cidade Universit\'{a}ria Zeferino Vaz, 13083--859, Campinas, SP,
Brazil.} \email{ddnovaes@unicamp.br} 



\maketitle

\begin{abstract}
The existence of a uniform upper bound for the maximum number of limit cycles of planar piecewise linear differential systems with two zones separated by a straight line has been subject of interest of hundreds of papers. After more than 30 years of investigation since  Lum--Chua's work, it has remained an open question whether this uniform upper bound exists or not. Here, we give a positive answer for this question by establishing the existence of a natural number $L^*\leq 8$ for which any planar piecewise linear differential system with two zones separated by a straight line has no more than $L^*$ limit cycles. The proof is obtained by combining a newly developed integral characterization of Poincar\'{e} half-maps for linear differential systems with an extension of Khovanski\u{\i}'s  theory for investigating the number of intersection points between smooth curves and a particular kind of orbits of vector fields.
\end{abstract}

\section{Introduction and statement of the main result}

The second part of the 16th Hilbert's Problem is one of the most important topics in the qualitative theory of planar differential systems (see, for instance, \cite{Ilyashenko02,Lloyd88}). Roughly speaking, given a positive integer $n,$ this problem inquires about the existence of a uniform upper bound $H(n)$ for the maximum number of limit cycles that planar polynomial differential systems of degree $n$ can have. Since linear differential systems do not admit limit cycles, we have $H(1)=0$. However, it remains unsolved whether $H(n)$ is finite, even for the simplest case $n=2$. 

The same problem has been also considered for planar nonsmooth differential systems. The study of limit cycles for such systems can be traced back to the work of Andronov et. al \cite{AndronovEtAl66} in 1937.  The simplest examples of planar nonsmooth differential systems are the planar piecewise linear differential systems with two zones separated by a straight line,
\begin{equation}\label{s1}
\dot \x =
\left\{\begin{array}{l}
A_L\x+\bb_L, \quad\textrm{if}\quad x_1\leq 0,\\
A_R\x+\bb_R, \quad\textrm{if}\quad x_1\geq 0.
\end{array}\right.
\end{equation}
Here, $\x=(x_1,x_2)\in\R^2,$ $A_{L,R}=(a_{ij}^{L,R})_{2\times 2},$ and  $\bb_{L,R}=(b_1^{L,R},b_2^{L,R})\in\R^2.$ The Filippov's convention \cite{Filippov88} is assumed for trajectories of \eqref{s1}.  In this context, a limit cycle is defined as an isolated crossing periodic solution.

The search for a uniform upper bound for the maximum number of limit cycles of differential systems of kind \eqref{s1} started some decades ago. Indeed, Lum and Chua \cite{LumChua91} in 1991, under the continuity hypothesis $a_{12}^L=a_{12}^R$, $a_{22}^L=a_{22}^R,$ and $\bb_L=\bb_R$,  conjectured that differential system \eqref{s1} had at most one limit cycle. This conjecture was first proven in 1998 by Freire et al. \cite{FreireEtAl98}. The next natural step was to relax the hypothesis of continuity. In 2010, Han and Zhang \cite{HanZhang10} proved the existence of piecewise linear differential systems of kind \eqref{s1} having two limit cycles. Based on their examples, they conjectured that such systems could have at most 2 limit cycles. In 2012, using numerical arguments, Huan and Yang  \cite{HuanYang12} gave a negative answer to this conjecture by showing an example with 3 limit cycles. In the same year, Llibre and Ponce \cite{LlibrePonce12} proved analytically the existence of such numerically observed limit cycles. After that, many other works provided examples with 3 limit cycles  (see, for instance, \cite{BuzziEtAl13,cardoso20,FreireEtAl14, FreireEtAl14b, LlibreEtAl15b,NovaesTorregrosa17}). 

Some partial results can be found in the literature regarding upper bounds for the maximum number of limit cycles for other non-generic families of piecewise linear differential systems (see, for instance, \cite{FreireEtAl12,LI2021101045,Li_2020,LlibreEtAl15,MedradoTorregrosa15,Novaes17}). However, up to now, after more than 30 years of investigation since the Lum--Chua's paper \cite{LumChua91} and hundreds of papers on this matter, it has remained an open question whether there exists or not a uniform upper bound for the maximum number of limit cycles that differential systems of kind \eqref{s1} can have. 

Recently, Carmona and Fern\'{a}ndez-S\'{a}nchez \cite{CarmonaEtAl19} obtained an integral characterization for {\it Poincar\'{e} half-maps} associated to a straight line of planar linear differential systems. This characterization was used in \cite{CarmonaEtAl19b} to provide a new simple proof for the Lum--Chua's conjecture. This last approach has proven to be an effective method to avoid the case-by-case study performed in the former proof in \cite{FreireEtAl98}. The same technique was used in \cite{Carmona2022-mc} to prove that differential system \eqref{s1}, under the assumption of nonexistence of a sliding set, has at most one limit cycle. 

Here, our main result provides a positive answer for the existence of a uniform upper bound for the maximum number of limit cycles that differential systems of kind \eqref{s1} can have.

 \begin{theorem}\label{main}
There exists a natural number $L^*\leq 8$ such that any planar piecewise linear differential system of  kind \eqref{s1} has no more than $L^*$ limit cycles. 
 \end{theorem}

Theorem \ref{main} is proven in Section \ref{sec:proof} by combining the integral characterization for Poincar\'{e} half-maps provided in \cite{CarmonaEtAl19} with an extension of Khovanski\u{\i}'s  theory for investigating the number of intersection points between smooth curves and a particular kind of orbits of vector fields. Section \ref{sec:kov} is dedicated to present  this extension of Khovanski\u{\i}'s  theory.

\section{Intersection between smooth curves with separating solutions}\label{sec:kov}

Khovanski\u{\i}, in \cite[Chapter II]{kho}, introduces the concept of {\it separating solutions} for vector fields defined on the whole plane. In his definition, an orbit of a vector field is a separating solution if it either is a cycle or corresponds to a noncompact trajectory that goes to and comes from infinity.

In the following result, Khovanski\u{\i} bounds the number of isolated intersection points between a given smooth curve and any orbit of a vector field that is a separating solution by means of the number of {\it contact points} between the curve and the vector field (that is, points of the curve in which the vector field is tangent to the curve at these points). Here, a smooth curve means a 1-dimensional $C^1$ submanifold of the plane (without boundary and possibly nonconnected). 

\begin{theorem}[{\cite[Corollary of Section 2.1]{kho}}]\label{thm:kov0} Consider a smooth vector field $X:\R^2\to\R^2$. Let a smooth curve  $\gamma\subset \R^2$  have at most $N$ noncompact (and any number of compact) connected components and have at most $k$ contact points with $X$. Then, there are at most $N+k$ isolated intersection points between $\gamma$ and any orbit of $X$ that is a separating solution.
\end{theorem}

In the present paper, we will apply an extension of this result for vector fields defined on open simply connected subsets of the plane, which is based on the fact that such subsets are diffeomorphic to the whole plane (as an application of the Riemann Mapping Theorem \cite{ahl} together with the fact that the unit disc is diffeomorphic to the plane). Accordingly, the definition above for an orbit to be a separating solution can be immediately extended for vector fields defined on simply connected open subsets of $\R^2$ as follows.

\begin{definition}\label{def:ss}
Consider a smooth vector field $X:U\to\R^2$ defined on an open simply connected subset $U\subset\R^2.$ An orbit $\mathcal{O}$ of the vector field $X$ is called a {\bf separating solution} if it is either a cycle or a noncompact trajectory satisfying $(\ov{\mathcal{O}}\setminus\mathcal{O})\subset\partial U$. Here, as usual, $\ov{\mathcal{O}}$ and $\partial U$ denote, respectively, the closure of $\mathcal{O}$ and the boundary of $U$ with respect to the $\R^2$ topology.
\end{definition}

 Let us provide some clarification on Definition \ref{def:ss}. Denote by  $\phi:U\to\R^2$ a diffeomorphism between $U$ and $\R^2$. When $\mathcal{O}$ is a cycle, $\phi(\mathcal{O})$ is also a cycle of the transformed vector field $\phi_* X:\R^2\to \R^2$ and Definition \ref{def:ss} agrees with the definition given by Khovanski\u{\i}. When $\mathcal{O}$ is noncompact, it is the image in $U$ of an open interval by an injective function and, thus, the condition $(\ov{\mathcal{O}}\setminus\mathcal{O})\subset\partial U$ implies that $\phi(\mathcal{O})$ is an orbit of the transformed vector field $\phi_* X$ that goes to and comes from infinity, i.e. a separating solution of $\phi_* X$ in the Khovanski\u{\i} sense.

The next result extends Theorem \ref{thm:kov0} to  orbits that are separating solutions of vector fields defined on open simply connected subset of the plane. Its proof, as mentioned before, follows immediately by transforming $U$ into the whole plane via a diffeomorphism and, then, applying Theorem \ref{thm:kov0}.

\begin{theorem}\label{thm:count}
Consider a smooth vector field $X:U\to\R^2$ defined on an open simply connected subset $U\subset\R^2.$ Let a smooth curve  $\gamma\subset \R^2$  have at most $N$ noncompact (and any number of compact) connected components and have at most $k$ contact points with $X$. Then, there are at most $N+k$ isolated intersection points between $\gamma$ and any orbit of $X$ that is a separating solution.
\end{theorem}

\section{Proof of the main result}\label{sec:proof}
This section is completely dedicated to the proof of Theorem \ref{main}.  We start by establishing a technical lemma ensuring that a uniform upper bound for the number of simple zeros of a 1-parameter family of analytic functions, under a suitable monotonicity condition on the parameter, also bounds the number of isolated zeros of functions in this family.

\begin{lemma}\label{lem:semi}
Let $I,J\subset\R$ be open intervals and consider a smooth function $\delta:I\times J\to \R$. Assume that
\begin{itemize}
\item[i.] for each $b\in J$, the function $\delta(\cdot,b)$ is analytic;
\item[ii.] there exists a natural number $N$ such that, for each $b\in J$, the number of simple zeros of the function $\delta(\cdot,b)$ does not exceed $N$; and 
\item[iii.] $\dfrac{\p \delta}{\p b}(u,b)>0$, for every $(u,b)\in I\times J$.
\end{itemize}
Then, for each $b\in J$, the function $\delta(\cdot,b)$ has at most $N$ isolated zeros.
\end{lemma}
\begin{proof}
For a fixed $\ov b\in J$, let $\ov u\in I$ be an isolated zero of $\delta(\cdot,\ov b)$. In this case, since $\delta(\cdot,\ov b)$ is analytic, there exist $\ov a\neq0$, a positive integer $\ov k,$ and an analytic function $R$ such that 
$
\delta(u,\ov b)=\ov a(u-\ov u)^{\ov k}+(u-\ov u)^{\ov k+1}R(u).
$

We start this proof by describing the unfolding of the isolated zero $\ov u$ in three distinct scenarios, namely: $(O)$ when $\ov k$ is odd; ($E^+$) when $\ov k$ is even and $\ov a>0;$ and ($E^-$) when $\ov k$ is even and $\ov a<0$. Taking into account condition (iii), we get the following unfolding in each scenario:
\begin{itemize}
\item If $\ov u$ satisfies $O$ (that is, $\ov k$ is odd), then there exists $\e>0$ sufficiently small such that the map $ \delta(\cdot,b)$ has a continuous branch of zeros, $u_0:(\ov b-\e,\ov b+\e)\subset J\to \R$, which are simple for $b\neq\ov b$ and $u_0(\ov b)=\ov u;$

\item If $\ov u$ satisfies $E^+$ (that is, $\ov k$ is even and $\ov a>0$), then there exists $\e>0$ sufficiently small such that the map $ \delta(\cdot,b)$ has two continuous branches of zeros, $u_1,u_2:(\ov b-\e,\ov b]\subset J\to\R$, which are simple for $b\neq\ov b$ and $u_1(\ov b)=u_2(\ov b)=\ov u;$

\item If $\ov u$ satisfies $E^-$ (that is, $\ov k$ is even and $\ov a<0$), then there exists $\e>0$ sufficiently small such that the map $ \delta(\cdot,b)$ has two continuous branches of zeros, $u_1,u_2:(\ov b,\ov b+\e]\subset J\to\R$, which are simple for $b\neq\ov b$ and $u_1(\ov b)=u_2(\ov b)=\ov u.$
\end{itemize}

Now, assume by absurd that there exists $b^*\in J$ such that the map $\delta(\cdot,b^*)$ has more than $N$  isolated zeros. Consider an amount of $N+1$ of these zeros and let $o,e^+,$ and $e^-$ be the number of such zeros
satisfying $O,E^+,$ and $E^-$, respectively. Clearly, $o+e^++e^-= N+1$. Let us assume, without loss of generality, that $e^+\geq e^-$. Thus, from the unfolding scenarios above we can choose a minimum $\e>0$ such that, for each $b\in(b^*-\e,b^*)\subset J$, the map $\delta(\cdot,b)$ has at least $o+2e^+\geq N+1$ simple zeros, which contradicts assumption (i). It concludes this proof.
\end{proof}

Now, before proving Theorem \ref{main}, we must set forth some preliminary concepts and results. Under the assumption $a_{12}^{L}a_{12}^{R}>0$ (which is necessary for the existence of limit cycles), Freire et. al in \cite[Proposition 3.1]{FreireEtAl12} provided that the differential system \eqref{s1} is transformed, by a homeomorphism preserving the separation line $\Sigma=\{(x,y)\in\R^2:\,x=0\}$, into the following Li\'enard canonical form 
\begin{equation}\label{cf}
\left\{\begin{array}{l}
\dot x= T_L x-y\\
\dot y= D_L x-a_L
\end{array}\right.\quad \text{for}\quad x< 0,
\quad 
\left\{\begin{array}{l}
\dot x= T_R x-y+b\\
\dot y= D_R x-a_R
\end{array}\right.\quad \text{for}\quad x> 0,
\end{equation}
where $a_{L}=a_{12}^{L}b_2^{L}-a_{22}^{L}b_1^{L},$ $a_{R}=a_{12}^{R}b_2^{R}-a_{22}^{R}b_1^{R},$ $b=a_{12}^Lb_1^R/a_{12}^R-b_1^L$, and $T_L,$ $T_R$ and $D_L,$ $D_R$ are, respectively, the traces and determinants of the matrices $A_L$ and $A_R$.

The periodic behavior of differential system \eqref{cf} can be analyzed by means of two Poincar\'{e} Half-Maps associated to $\Sigma$, namely, the {\it Forward Poincar\'{e} Half-Map}  $y_L: I_L\subset [0,+\infty) \longrightarrow(-\infty,0]$ and  the {\it Backward Poincar\'{e} Half-Map} $y_R^b:I_R^b\subset [b,+\infty)\rightarrow (-\infty,b]$. The forward one maps a point $(0,y_0)$, with $y_0\geq0$, to a point $(0,y_L(y_0))$ by following the flow in the positive direction. Analogously, the backward one maps a point $(0,y_0)$, with $y_0\geq b$, to $(0,y_R^ b(y_0))$ by following the flow in the negative direction. Notice that the left differential system defines $y_L$ and the right differential system defines $y_R^b$. The map  $y_L$ is characterized by an integral relationship provided by Theorem 19, Corollary 21, and Remark 24 of \cite{CarmonaEtAl19}. The map $y_R^b$ can also be characterized via such results just by considering a change of variables and parameters. It is worthwhile to mention that $y_R^b(y_0)=y_R^0(y_0-b)+b$ and $I_R^b=I_R^0+b$, where $y_R^0:I_R^0\subset [0,+\infty)\rightarrow (-\infty,0]$ is the  {\it Backward Poincar\'{e} Half-Map} of \eqref{cf} for $b=0$. For the sake of simplicity, let us denote $y_R^0$ and $I_R^0$ just by $y_R$ and $I_R$, respectively.

Important properties, described in \cite{CarmonaEtAl19}, of the maps $y_L$ and $y_R$ are obtained from the following polynomials
\begin{equation}\label{eq:poly}
W_L(y)=D_Ly^2-a_LT_Ly+a_L^2\quad\text{and}\quad W_R(y)=D_Ry^2-a_RT_Ry+a_R^2.
\end{equation}
Indeed, the graphs of $y_L$ and $y_R$, oriented according to increasing $y_0$, are, respectively, the portions included in the fourth quadrant of particular orbits of the 
cubic vector fields
\begin{equation}\label{dy_1L}
X_L(y_0,y_1)=-\big(y_1W_L(y_0) ,y_0 W_L(y_1)\big) \quad\text{and}\quad X_R(y_0,y_1)=-\big(y_1W_R(y_0) ,y_0 W_R(y_1)\big).
\end{equation}
In addition, the curves $y_L(y_0)$ and $y_R(y_0)$ are, respectively, solutions of the differential equations
 \begin{equation}
\label{eq:odeL}
\dfrac{d y_1}{d y_0}= \dfrac{y_0W_L(y_1)}{y_1W_L(y_0)} \quad \text{and}\quad  \dfrac{d y_1}{d y_0}= \dfrac{y_0W_R(y_1)}{y_1W_R(y_0)} .
\end{equation}

\begin{remark}\label{rem:prop}
\label{remark:intervlas} The polynomials given in \eqref{eq:poly} also provide information regarding the intervals $I_L$ and $I_R$ of definition of $y_L$ and $y_R$, respectively.  The smallest positive root of  $W_L$, if any, is the right endpoint of $I_L$. Analogously, the greatest negative root of  $W_L$, if any, is the left endpoint of $y_L(I_L)$. When $4D_L-T_L^2>0$, since the polynomial $W_L$ has no roots, the intervals  $I_L$  and $y_L(I_L)$  are unbounded with $y_L(y_0)$  tending to $-\infty$ as $y_0\to +\infty$.  The polynomial $W_L$  is strictly positive in  $[y_L(y_0),0 )\cup (0,y_0]$,  for $y_0\in I_L$. The same conclusions are valid for $y_R$. 
\end{remark}

Another property which will be important in this proof concerns about the relative position between the graph of the forward Poincar\'{e} half-map and the bisector of the fourth quadrant (see \cite{Carmona2021-ru} for a proof). A similar result can be given for the backward Poincar\'{e} half-map.

\begin{proposition}
\label{rm:signoy0+y1}
The relationship $\sgn\left(y_0+y_L(y_0) \right)=-\sgn(T_L)$ holds for $y_0\in  I_L\setminus\{0\}$.
\end{proposition} 

Now, we can proceed with the proof of Theorem \ref{main}.

As usual, crossing periodic solutions of \eqref{cf} are studied by means of the displacement function $\delta_b$, which is defined in the interval $I_b:=I_L\cap(I_R+b)$ as follows
\[
\delta_b(y_0)=y_R(y_0-b)+b-y_L(y_0).
\]
Indeed, the zeros of $\delta_b$ in $\operatorname{int}(I_b)$ are in bijective correspondence with crossing periodic solutions of \eqref{cf} as well as simple zeros $\delta_b$ in $\operatorname{int}(I_b)$ are in bijective correspondence with hyperbolic limit cycles of \eqref{cf}. 

In light of Lemma \ref{lem:semi}, we will show that, for each $b\in\R$, the number of simple zeros of $\delta_b$ does not exceed $8$.

Taking into account the derivatives of $y_L$ and $y_R$ given in \eqref{eq:odeL}, one can easily see that, if $y_0^*\in \operatorname{int}(I_b)$ satisfies $\delta_b(y_0^*)=0$, then 
\[
\delta_b'(y_0^*)=\dfrac{(y_0^*-y_1^*)}{y_1^*(y_1^*-b)W_L(y_0^*)W_R(y_0^*-b)}F_b(y_0^*,y_1^*),
\]
where $y_1^*=y_R(y_0^*-b)+b=y_L(y_0^*)<\min(0,b)$ and $F_b$ is a polynomial function of degree 4 given by
\[
F_b(y_0,y_1)=m_0 + m_1 (y_0 + y_1) + m_2 y_0 y_1 + m_3 (y_0^2 + y_1^2) + 
 m_4 (y_0 y_1^2 + y_0^2 y_1) + m_5 y_0^2 y_1^2,
\]
with $m_i,$ $i=1,\ldots,5,$ being polynomial functions on the parameters of differential system \eqref{cf}.
%\[
%\begin{aligned}
%m_0 =& a_L^2 (a_R^2 b + b^3 D_R + a_R b^2 T_R),\\
%m_1 = &-a_L^2 b (2 b D_R + a_R T_R) ,\\
%m_2 = &(-a_R^2 (b D_L + a_L T_L) + b D_R (3 a_L^2 - b^2 D_L + a_L b T_L) + 
%    a_R (a_L^2 - b^2 D_L) T_R) ,\\
%m_3 =& a_L^2 b D_R ,\\
%m_4 = &(a_R^2 D_L - D_R (a_L^2 - b^2 D_L + a_L b T_L) + a_R b D_L T_R) ,\\
%m_5 = &(-b D_L D_R + a_L D_R T_L - a_R D_L T_R).
%\end{aligned}
%\]
 
 From Remark \ref{rem:prop}, $W_L(y_0^*)W_R(y_0^*-b)>0$ and, since $(y_0^*-y_1^*)y_1^*(y_1^*-b)>0$, thus
$\sgn(\delta'_b(y_0^*))=\sgn(F_b(y_0^*,y_1^*)).$
This means that the zero set $\gamma_b=F_b^{-1}(\{0\})$ separates the attracting hyperbolic crossing limit cycles from the repelling ones. 
Consequently, since two consecutive hyperbolic limit cycles of \eqref{cf} cannot have the same stability,  the number of them  and, therefore, the number of simple zeros of $\delta_b$ is bounded by the number of isolated intersection points between $\gamma_b$ and one of the curves $y_1=y_{L}(y_0)$ or $y_1=y_{R}(y_0-b)+b,$ $y_0\in \Int(I_b)$, increased by one.

It is sufficient to focus our attention to the curve $\mathcal{O}_b=\{(y_0,y_L(y_0)):\,y_0\in \Int (I_b)\}$. 
From Proposition \ref{rm:signoy0+y1}, one of the following cases holds:
\begin{enumerate}[(i)]
\item  $T_L=0$, then $\mathcal{O}_b \subset \{(y_0,-y_0):\,y_0>0\}$;
\item $T_L< 0$, then $\mathcal{O}_b \subset B^+:=\{(y_0,y_1):\, -y_0< y_1< 0\}$;
\item $T_L> 0$, then $\mathcal{O}_b \subset B^-:=\{(y_0,y_1):\,-y_1>y_0>0\}$.
\end{enumerate}

 In what follows, we are going to show that the number of isolated intersection points between $\gamma_b$ and $\mathcal{O}_b$ is at most $7$. First, for $T_L=0$, since from (i) $\mathcal{O}_b$ is a straight segment, by Bezout's theorem the number of isolated intersection points between $\gamma_b$ and $\mathcal{O}_b$ is at most $4$. Thus, from now on, we assume that $T_L\neq0$. In this case, $\mathcal{O}_b$ is a separating solution of the restricted vector field 
\begin{equation}\label{XLU}
\widehat X_L=X_L\big|_U\quad \text{with}\quad U=B\cap\Int\big(I_b\times(y_L(I_b)\cap y_R^b(I_b))\big),
\end{equation}
where $X_L$ is given in \eqref{dy_1L} and, by taking into account cases (ii) and (iii),
$B$ is either $B^+$ or $B^-$ provided that $T_L< 0$ or $T_L> 0$, respectively. Notice that, since $I_b$, $y_L(I_b)$, and $y_R^b(I_b)$ are intervals, then $U$ is an open simply connected subset of the quadrant $Q:=\{(y_0,y_1):\,y_0>0\,\text{ and }\, y_1<0\}$.

In order to use Theorem \ref{thm:count} for bounding the number of isolated intersection points between $\gamma_b$ and $\mathcal{O}_b$, we have to estimate the number of contact points between $\gamma_b$ and $\widehat X_L$, which can be done by means of the inner product
\[
G_b(y_0,y_1)=\langle\nabla F_b(y_0,y_1),X_{L}(y_0,y_1)\rangle.
 \] 
 One can see that $G_b$ is a polynomial function of degree $6$ given by
 \[
\begin{aligned}
G_b(y_0,y_1)
=& n_1 (y_0 + y_1) + n_2  y_0 y_1 + n_3 (y_0^2 + y_1^2) + n_4 (y_0^2 y_1 + y_0 y_1^2) +
  n_5 (y_0^3 + y_1^3)\\
  & + n_6 y_0^2 y_1^2 + n_7 (y_0^3 y_1 + y_0 y_1^3) + 
 n_8 (y_0^3 y_1^2 + y_0^2 y_1^3) + n_9 y_0^3 y_1^3,
\end{aligned}
 \] 
 where $n_i,$ $i=1,\ldots 9,$ are polynomial functions on the parameters of differential system \eqref{cf}.
% \[
% \begin{aligned}
% n_1 =& a_L^4 b (2 b D_R + a_R T_R) ,\\
%n_2 =& -2 a_L^3 b (2 a_L D_R + 2 b D_R T_L + a_R T_L T_R) ,\\
%n_3 =& a_L^2 (a_R^2 (b D_L + a_L T_L) + b D_R (-3 a_L^2 + b^2 D_L - a_L b T_L) + 
%     a_R (-a_L^2 T_R + b^2 D_L T_R)) ,\\
%n_4 =& a_L (2 a_L^3 D_R + a_L^2 T_L (7 b D_R + a_R T_R) - 
%     b D_L T_L (a_R^2 + b^2 D_R + a_R b T_R) \\&- 
%     a_L (-b^2 D_R T_L^2 + a_R^2 (2 D_L + T_L^2) + a_R b D_L T_R)) ,\\
%n_5 =& a_L^2 (-a_R^2 D_L + D_R (a_L^2 - b^2 D_L + a_L b T_L) - a_R b D_L T_R) ,\\
%n_6 =& 2 (a_R^2 D_L (b D_L + 3 a_L T_L) + 
%     D_R (b^3 D_L^2 - 2 a_L^3 T_L + a_L b^2 D_L T_L - 
%        a_L^2 b (3 D_L + 2 T_L^2))\\& + 
%     a_R D_L (-a_L^2 + b^2 D_L + 2 a_L b T_L) T_R) ,\\
%n_7 =& a_L (a_R^2 D_L T_L + D_R T_L (-3 a_L^2 + b^2 D_L - a_L b T_L) + 
%     a_R D_L (2 a_L + b T_L) T_R) ,\\
%n_8 =& -3 a_R^2 D_L^2 + 
%     D_R (-3 b^2 D_L^2 + a_L b D_L T_L + a_L^2 (3 D_L + 2 T_L^2)) - 
%     a_R D_L (3 b D_L + 2 a_L T_L) T_R ,\\
%n_9 =& 4 D_L (b D_L D_R - a_L D_R T_L + a_R D_L T_R) .
% \end{aligned}
% \]
  Accordingly, the number of contact points between $\gamma_b$ and $\widehat X_L$ is bounded by the number of isolated solutions of the polynomial system 
\begin{equation}\label{eq:system}
 F_b(y_0,y_1)=0 \,\text{ and }\,  G_b(y_0,y_1)=0,\,\,(y_0,y_1)\in U.
 \end{equation}

Now, since we are only interested in solutions of \eqref{eq:system} in $U$, we proceed with the following change of variables
 \[
(Y_0,Y_1)=\phi(y_0,y_1):=(y_0+y_1,y_0 y_1)\,\text{ for }\, (y_0,y_1)\in U,
\]
which is a diffeomorphism between the quadrant $Q$ and the half-plane $\{(Y_0,Y_1):Y_1<0\}$ that transforms the polynomial system \eqref{eq:system} into the following equivalent one
\begin{equation}\label{eq:system2}
\widetilde F_b(Y_0,Y_1)=0 \,\text{ and }\,  \widetilde G_b(Y_0,Y_1)=0,\,\, (Y_0,Y_1)\in\phi(U),
 \end{equation}
where, now, $\widetilde F_b$  and $\widetilde G_b$ are polynomial functions of degrees $2$ and $3$ given, respectively, by
\[
\begin{aligned}
\widetilde F_b(Y_0,Y_1)=& m_0 + m_1 Y_0 + (m_2 - 2 m_3) Y_1+ m_3 Y_0^2  + m_4 Y_0 Y_1 + m_5 Y_1^2,\\
 \widetilde G_b(Y_0,Y_1)=&n_1 Y_0+ (n_2 - 2 n_3) Y_1 + n_3 Y_0^2 + (n_4 - 3 n_5) Y_0 Y_1+ (n_6 - 2 n_7) Y_1^2+ n_5 Y_0^3 + n_7 Y_0^2 Y_1  + n_8 Y_0 Y_1^2 + n_9 Y_1^3.
\end{aligned}
\]

We claim that \eqref{eq:system2} has at most $5$ finite isolated solutions. Notice that, by Bezout's Theorem, polynomial system \eqref{eq:system2} has at most $6$ isolated solutions (finite or not).
For $D_L=0$, one can see that the polynomial system $\widetilde F_b(Y_0,Y_1)=0 \,\text{ and }\,  \widetilde G_b(Y_0,Y_1)=0$ has a solution at the infinity, decreasing the number of possible finite isolated solutions of 
\eqref{eq:system2} at least by $1$. 
Thus, it remains to analyze what happens for $D_L\neq0$. In this case, one can  check that
 \[
 (Y_0^*,Y_1^*)=\left(\dfrac{a_L T_L}{D_L},\dfrac{a_L^2}{D_L}\right)
 \]
is a solution of $\widetilde F_b(Y_0,Y_1)=0 \,\text{ and }\,  \widetilde G_b(Y_0,Y_1)=0$. Let us see that  $ (Y_0^*,Y_1^*)\notin\phi(U)$. On the one hand, depending on the sign of $T_L$, the set $\phi(U)$ satisfies that: (j) if $T_L> 0$, $\phi(U)\subset \{(Y_0,Y_1):Y_0<0,Y_1<0\}$; and (jj)
if $T_L<0$, $\phi(U)\subset \{(Y_0,Y_1):Y_0>0,Y_1<0\}$.
On the other hand, in order to determine the relative position of $(Y_0^*,Y_1^*)$ with respect to $\phi(U)$ we have to consider some cases. If $D_L>0$ or $a_L=0$, then  $Y_1^*\geq0$, which implies that $(Y_0^*,Y_1^*)\notin\phi(U)$. 
 If $D_L<0$ and $a_L<0$, then $\sgn(Y_0^*)=\sgn(T_L)$ and, consequently, from cases (j) and (jj) above, $(Y_0^*,Y_1^*)\notin\phi(U)$. Finally, if $D_L<0$ and  $a_L>0$, then the quadratic polynomial $W_L$ has two simple roots, $y_1^*<0<y_0^*$. One can easily seen that $(Y_0^*,Y_1^*)=\phi(y_0^*,y_1^*)$. From Remark \ref{rem:prop}, $y_0^*\notin\Int(I_L)$ and $y_1^*\notin\Int(y_L(I_L))$, then from the definition of $U$ in \eqref{XLU}, $(y_0^*,y_1^*)\notin U$ and, consequently, $(Y_0^*,Y_1^*)\notin\phi(U)$. 
 
Finally, since $\widetilde F_b$ has degree 2, we have that $\widetilde \gamma_b=\widetilde F_b^{-1}(\{0\})$ has at most two noncompact connected components in $\phi(U)$, which implies that $\gamma_b$ has at most $2$ noncompact connected components in $U$.
 
 Applying Theorem \ref{thm:count}, we conclude that $\gamma_b$ and $\mathcal{O}_b$ has at most $5+2=7$ intersections. Therefore, differential system \eqref{cf} has at most $8$ hyperbolic limit cycles and, consequently,  the number of simple zeros of the displacement function $\delta_b$ does not exceed $8$.

From here, we wish to apply Lemma \ref{lem:semi} to conclude that $\delta_b$ has at most $8$ isolated zeros for every $b\in\R$. However, it cannot be directly applied to $\delta_b$ because its domain depends on the parameter $b$. Thus, we proceed by contradiction as follows.
Assume that, for some $b^*\geq0$, $\delta_{b^*}$ has more than $8$ isolated zeros. Take an open interval $I\subset\Int( I_{b^*})$ that contains all the isolated zeros of $\delta_{b^*}$. Since $I_b=I_R^0+b$, we can consider a small interval $J$ containing $b^*$ for which $I\subset \Int( I_b)$ for every $b\in J$. Since $\delta_b|_I$ has at most $8$ simple zeros for every $b\in J$ and 
\[
\dfrac{\p}{\p b}\delta_b(y_0)=-y_R'(y_0-b)+1>0,
\,\text{ for }\, y_0\in I \,\text{ and }\, b\in J,
\]
Lemma \ref{lem:semi} implies that $\delta_b|_I$ has at most $8$ isolated zeros for every $b\in J$, which contradicts the initial assumption that $\delta_{b^*}$ has more than $8$ isolated zeros. It concludes the proof of Theorem \ref{main}.

\section*{Acknowledgements}

VC is partially supported by the Ministerio de Ciencia, Innovaci\'on y Universidades, Plan Nacional I+D+I cofinanced with FEDER funds, in the frame of the project PGC2018-096265-B-I00. FFS is partially supported by the Ministerio de Econom\'{i}a y Competitividad, Plan Nacional I+D+I cofinanced with FEDER funds, in the frame of the project MTM2017-87915-C2-1-P. VC and FFS are partially supported by the Ministerio de Ciencia e Innovaci\'on, Plan Nacional I+D+I cofinanced with FEDER funds, in the frame of the project PID2021-123200NB-I00, the Consejer\'{i}a de Educaci\'{o}n y Ciencia de la Junta de Andaluc\'{i}a (TIC-0130, P12-FQM-1658) and by the Consejer\'{i}a de Econom\'{i}a, Conocimiento, Empresas y Universidad de la Junta de Andaluc\'{i}a (US-1380740, P20-01160). DDN is partially supported by S\~{a}o Paulo Research Foundation (FAPESP) grants 2022/09633-5, 2021/10606-0, 2019/10269-3, and 2018/13481-0, and by Conselho Nacional de Desenvolvimento Cient\'{i}fico e Tecnol\'{o}gico (CNPq) grants 438975/2018-9 and 309110/2021-1.

%\bibliographystyle{abbrv}
%\bibliography{references.bib}

\end{document}